\documentclass[10pt]{amsart}
\usepackage{amsmath}
\usepackage{amssymb,amscd}
\usepackage{graphicx}
\usepackage{tikz}
\usepackage{tikz-cd}
\usepackage{mathdots}
\usepackage{mathrsfs}       % for the script X used by \XXS
\usepackage{graphicx}

\usepackage{microtype} % to adjust spacing between characters to reduce the number of bad line breaks.

\usepackage{colonequals} %\colonequal

\theoremstyle{plain}
\newtheorem{theorem}{Theorem}[section]
\newtheorem{lemma}[theorem]{Lemma}
\newtheorem{proposition}[theorem]{Proposition}
\newtheorem{corollary}[theorem]{Corollary}
\newtheorem{question}[theorem]{Question}

\theoremstyle{definition}
\newtheorem{definition}[theorem]{Definition}
\newtheorem{example}[theorem]{Example}

\numberwithin{equation}{section}

\def \begineq{\begin{equation}}
\def \endeq{\end{equation}}

\def \bb{\mathbb}

\def \CC{{\bb{C}}}
\def \CP{\bb{C}{\rm P}}

\def \RR{{\bb{R}}}
\def \ZZ{{\bb{Z}}}

\def \Ann{{\rm Ann}}
\def \Aut{{\rm Aut}}

\def \Hom{{\rm Hom}}

\def \rk{{\rm rk}}
\def \im{{\rm im}}

\begin{document}

\title[Locally $k$-standard $T$-manifolds]{Equivariant cohomological rigidity of  certain $T$-manifolds}

\author[S. Sarkar]{Soumen Sarkar}

\address{Department of Mathematics, Indian Institute of Technology Madras,
Chennai 600036, India}
\email{soumen@iitm.ac.in}

\author[J. Song]{Jongbaek Song}
\address{School of Mathematics, KIAS, Seoul 02455, Republic of Korea}
\email{jongbaek@kias.re.kr}

\date{\today}


%\thanks{}

\abstract  
We introduce the category of {\it locally $k$-standard $T$-manifolds} which includes well-known classes of manifolds such as toric and quasitoric manifolds, good contact toric manifolds and moment-angle manifolds. They are smooth manifolds with well-behaved actions of tori. We study their topological properties, such as fundamental groups and equivariant cohomology algebras. Then, we discuss when the torus equivariant cohomology algebra distinguishes them up to weakly equivariant homeomorphism. 
\endabstract

\maketitle

\section{Introduction}\label{intro}
%Contact toric manifolds were introduced in the study of Hamiltonian torus
%actions on odd-dimensional smooth compact manifolds in \cite{BaMo, BoGa}.
%These are odd-dimensional analogues of symplectic toric
%manifolds with Hamiltonian torus actions.
%The papers \cite{ToZe} and \cite{Lerman} worked on the geometric aspects like integrable
%geodesic flows on contact toric manifolds. Moreover, Lerman  \cite[Theorem 2.18]{Ler} provides a complete classification of compact connected contact toric manifolds. Also, Luo in his thesis \cite{Luo-T} studies several topological aspects of contact toric manifolds.

The interaction among algebra,  combinatorics, geometry and topology  is blossoming in the garden of mathematics. It often comes through group actions on topological or geometric objects. Depending on the hypothesis on group actions on spaces, one associates various types of combinatorial objects to analyze their algebraic, geometric and topological properties. 
Amongst a variety of examples along the lines of this philosophy, the spaces equipped with torus actions have been intensively studied for the last few decades and  the following interesting bridges have been found. 
 
 Symplectic manifolds with Hamiltonian torus actions have convex polytopes as their images of moment maps \cite{Atiyah, GuSt}. In addition, if the dimension of the torus is the half of the dimension of a given manifold, then the image of the moment map is given by a Delzant polytope, and such a correspondence is bijective up to some equivalent relations \cite{Del}. GKM manifolds are related to certain graphs with labels on edges encoding the tangential representation at each fixed point \cite{GKM, GuZa}. A normal separated toric variety defines a collection of cones in a vector space, which is called a fan, and vice versa
\cite{CLS, Ful, Oda}. Also, the pioneering paper \cite{DJ} initiated a topological generalization of smooth projective toric varieties
and studies them using the combinatorics of orbit spaces and group action data. 
These are even dimensional spaces with effective torus actions with nonempty fixed points such that dimensions of tori are less than or equal to the half of the dimensions of the given spaces. We refer to \cite{BuTe, BuTe2} for the theory of $(2n, k)$-manifolds which puts together a wide class of $2n$-dimensional manifolds $M$ endowed with actions of $k$-dimensional tori $T$ such that $2\dim T \leq \dim M$. The manifolds listed above are examples of $(2n,k)$-manifolds. 

On the other hand, there are classes of manifolds beyond the above list of examples. 
%with torus action such that $2\dim T > \dim M.$  
For instance, compact connected contact toric manifolds of dimension greater than 3 with non-free torus actions are classified by cones satisfying certain conditions \cite{Lerman}. Moment-angle complexes and partial quotients have strong relationships with simplicial complexes \cite{Fr-PQ, BP-book}. We note that these examples are spaces $M$ with effective actions of tori $T$ such that $2\dim T > \dim M$ and   they have no fixed points. 

In this paper, we generalize some of the classes of manifolds listed above and call them \emph{locally $k$-standard $T$-manifolds}, which are $(2n+k)$-dimensional manifolds with effective actions of $(n+k)$-dimensional tori satisfying certain local property. The idea of a locally $k$-standard $T$-manifold is motivated by the work of Davis--Januszkiewicz \cite{DJ}, where they consider the standard $T^n$-action on $\CC^n$ as a local model to define toric manifolds which are also called quasitoric manifolds in recent literature. Here, we consider $T^{n+k}$-action on $\CC^{n}\times T^k$ which is an invariant subset of $\CC^{n+k}$ with respect to the standard $T^{n+k}$-action and adopt this as the local model to define a locally $k$-standard $T$-manifold, see Definition \ref{def_axiom}. This category of manifolds includes toric and quasitoric manifolds \cite{DJ, BP},  locally standard torus manifolds \cite{MP}, compact connected contact toric manifolds associated with good cones \cite{Lerman} and moment-angle manifolds \cite{BP}, as well as  infinitely many objects outside these categories which might be interesting in nature.

In this paper, we are primarily interested in locally $k$-standard $T$-manifolds whose orbit spaces are simple polytopes. We study their topological properties and equivariant classification via equivariant cohomology (or Borel equivariant cohomology). For a $G$-space $X$, the equivariant cohomology $H^\ast_G(X)$ is defined by the cohomology of the homotopy quotient, namely 
$$H_G^*(X) \colonequals H^*(EG \times_G X),$$
where $EG$ is the total space of the universal principal $G$-bundle and $EG \times_G X$ denotes the orbit space of the diagonal $G$-action on $EG \times X$.
The equivariant collapsing map $ X \to \{pt\}$ induces an $H^*(BG)$-algebra structure on $H_G^*(X)$, which carries various information about the $G$-action on $X$. 

Indeed, for the case of quasitoric manifolds or toric hyper-K\"{a}hler manifolds, somewhat surprisingly, the equivariant cohomology algebra distinguishes them up to weakly equivariant homeomorphisms or diffeomorphisms, see \cite{Kur, Mas}. Here, a weakly equivariant homeomorphism means a homeomorphism $\Psi \colon X\to Y$ between two $G$-spaces $X$ and $Y$ together with an automorphism $\delta$ of $G$ such that $\Psi(g \cdot x) = \delta(g) \cdot \Psi(x)$ for any $g\in G$ and $x\in X$. A weak isomorphism between two $R$-algebras for a ring $R$ is defined similarly. 
Now, one may ask the following:
\begin{question}
What are $G$-spaces $X$ and $Y$ such that they are weakly equivariantly homeomorphic, whenever their equivariant cohomologies are weakly isomorphic as $H^\ast(BG)$-algebras?
\end{question}
\noindent We call this question the \emph{equivariant cohomological rigidity} and 
in this paper, we give an affirmative answer to this question for a wide class of  locally $k$-standard $T$-manifolds.

%The article is organized as follows.

 We begin Section \ref{sec_tctm} with the axiomatic and the constructive definition of a   locally $k$-standard $T$-manifold and show these two definitions are equivalent when its orbit space is a simple polytope, see Corollary \ref{cor_two_def_are_equiv}. Then we observe when a   locally $k$-standard $T$-manifold can be constructed as a quotient of a moment-angle manifold, see Proposition \ref{prop_axion=const}. We investigate their fundamental groups in Lemma \ref{lem_H^1_finite_group}, which extends the result of \cite[Theorem 1.1]{Ler2}. This lemma helps to prove the main theorem of Section  \ref{sec_equiv_cohom_rigidity}.
  We also provide several categories of manifolds which can be understood as locally $k$-standard $T$-manifolds.
  
Section \ref{sec_equiv_coh_of_tctm} is devoted to studying the equivariant cohomology algebra of a locally $k$-standard $T$-manifold $M$. If a $T$-manifold has fixed points  like GKM-manifolds, then one can use the localization technique to analyze the equivariant cohomology. Indeed, this is the case for 
the equivariant cohomological rigidity theorems given in \cite{Mas} and \cite{Kur}. However, a locally $k$-standard $T$-manifold may not have fixed points if $k \geq 1$.
 Hence, it is difficult to
use the classical localization results to analyze the equivariant cohomology of $M$.
Fortunately, using the invariant submanifolds corresponding to the faces of the orbit
space, we can determine the generators and relations for the equivariant cohomology
ring $H^{\ast}_{T} (M)$.
% Hence, instead of localizing equivariant cohomology to fixed points, we use invariant submanifolds and their normal bundle which has a nice decomposition arising from the local $k$-standardness. 
%corresponding to the faces of the orbit space, we can determine the generators and relations for the equivariant cohomology ring $H^\ast_{T^{n+k}}(M)$. 
Then, under a mild hypothesis, we realize the ring  $H^\ast_{T}(M)$ as the Stanley--Reisner ring of the orbit space of $M$, see Theorem \ref{thm_equiv=SR}.  We also investigate  $H^\ast(BT)$-algebra structure of $H^\ast_{T}(M)$ which turns out to encode the complete information about the $T$-action on $M$.

Finally, in Section \ref{sec_equiv_cohom_rigidity}, we give a concrete answer to the equivariant cohomological rigidity about  locally $k$-standard $T$-manifolds which generalizes the result of \cite[Section 4]{Mas}.

For convenience, we often use the following well-known identifications without explicitly mentioning it.
\begin{itemize}
\item $\ZZ^{n+k} \cong \mathfrak{t}_\ZZ \cong H_2(BT^{n+k}) \cong \Hom(S^1, T^{n+k})$;
\item $(\ZZ^{n+k})^\ast \cong \mathfrak{t}_\ZZ^\ast \cong H^2(BT^{n+k}) \cong \Hom(T^{n+k}, S^1)$,
\end{itemize}
where $\mathfrak{t}_\ZZ = \ker( \exp \colon \mathfrak{t} \to T^{n+k})$ and all
cohomologies in this paper are considered with integer coefficients.

%%%%%%%%%%%%%%%%

%%%%%%%%%%%%%%%%%

\section{Locally $k$-standard $T$-manifolds}\label{sec_tctm}
Let $T$ be a torus. We consider manifolds with $T$-actions, which are denoted by $T$-manifolds. 
In this section, we introduce the category of \emph{locally $k$-standard $T$-manifolds} and study their essential properties which can be encoded by some combinatorial data. This category of manifolds generalizes the concept of toric manifolds introduced in \cite{DJ} by extending the idea of \emph{locally standard} $T$-action for a toric manifold.

\subsection{Axiomatic definition}\label{def_axiom}
Consider the action  
$\alpha \colon T^{n+k} \times \CC^{n+k} \to \CC^{n+k}$ of $(n+k)$-dimensional torus $T^{n+k}$ on $\CC^{n+k}$ defined by
$$\alpha((t_1, \ldots,t_n, \dots,  t_{n+k}), (z_1, \ldots,z_n, \dots,   z_{n+k}))= (t_1z_1, \ldots, t_nz_n, \dots,  t_{n+k} z_{n+k}).$$ 
Then the set $\CC^n \times T^k$
is a $T^{n+k}$-invariant subset of $\CC^{n+k}$, and the orbit space $(\CC^{n}\times T^k)/T^{n+k}$ is $\RR^{n}_\geq$, the positive orthant.  We call the restriction $\alpha|_{T^{n+k}\times (\CC^{n}\times T^k)}$ the \emph{$k$-standard} $T^{n+k}$-action on $\CC^{n}\times T^k$. 

\begin{definition}\label{def:axiom_top_cont}
A $(2n+k)$-dimensional smooth manifold $M$ with an effective $T^{n+k}$-action is called a \emph{locally $k$-standard $T$-manifold} if it is \emph{locally isomorphic} to $\mathbb{C}^n \times T^k$ with the $k$-standard $T^{n+k}$-action. Here, `locally isomorphic' means for each point $p$ of $M$, there is
\begin{enumerate}
\item  an automorphism $\theta_p\in {\rm Aut}(T^{n+k})$;
\item  a $T^{n+k}$-invariant neighborhood $U \subseteq M$ of $p$ which is $\theta_p$-equivariantly diffeomorphic to a $T^{n+k}$-invariant subset $V \subseteq \CC^n \times T^k$.
\end{enumerate}
\end{definition}

Since $T^{n+k}$-action is locally $k$-standard and transversality is
a local property, we get that the orbit space $M/T^{n+k}$ is a nice
manifold with corners of dimension $n$. In this paper, we are
primarily interested in the  locally $k$-standard $T$-manifolds 
whose orbit spaces are simple polytopes. Basic properties of 
simple polytopes can be found in \cite{Zi, BP-book}. 
%In particular, if $k=0$ in Definition \ref{def:axiom_top_cont} then it is called a locally standard torus manifold of \cite[]{MP}.  Section 4

Let $\mathfrak{q} \colon M \to P $ be the orbit map where $P$
is an $n$-dimensional simple polytope. Let $\mathcal{F}(P)\colonequals \{F_1, \ldots, F_m\}$ be the set of  codimension-1 faces, called facets, of $P$.
Then each $M_i \colonequals \pi^{-1}(F_i)$
is a $(2(n-1)+k)$-dimensional $T^{n+k}$-invariant submanifold of $M$. From the locally
$k$-standardness and \cite[Lemma 1.3]{DJ}, we can show that $M_{i}$ is
 a $(2(n-1)+k)$-dimensional  locally $k$-standard $T$-manifold over $F_i$. Therefore, the isotropy subgroup of $M_{i}$ is a circle subgroup $T_{i}$ of
$T^{n+k}$. The group $T_{i}$ is uniquely determined by a primitive vector
$\lambda_i \in \ZZ^{n+k}$. That is, we get a natural function
\begin{equation}\label{eq_axiomatic_lambda}
 \lambda \colon \{F_1, \ldots, F_m\} \to \ZZ^{n+k}   
\end{equation}
defined by $\lambda(F_i) = \lambda_i$.

Since each vertex $v$ of $P$ is the transversal intersection
of $n$ facets $\{F_{i_1}, \cdots, F_{i_n}\}$, the manifolds
$M_{i_1}, \ldots, M_{i_n}$ intersect transversely by the locally $k$-standardness.
This implies that the submodule $A$ of $\ZZ^{n+k}$ generated by $\{\lambda_{i_1}, \dots, \lambda_{i_n}\}$ corresponding to the $n$-dimensional subtorus $T_{i_1}\times \cdots \times T_{i_n}$ of $T^{n+k}$ is a direct summand of $\ZZ^{n+k}$. Indeed, there exists a primitive vectors $\lambda_{i_{n+1}}, \ldots, \lambda_{i_{n+k}} \in \ZZ^{n+k}$ such that the rank of $A \oplus \left< \lambda_{i_{n+1}}, \ldots, \lambda_{i_{n+k}} \right>$ is $n+k$ and the volume determined by $\{\lambda_{i_1}, \ldots, \lambda_{i_n}, \lambda_{i_{n+1}}, \ldots, \lambda_{i_{n+k}}\}$ in $\RR^{n+k}$
 is equal to the volume determined by $\{\lambda_{i_1}, \ldots, \lambda_{i_n}\}$ 
 in $A \otimes_{\ZZ} \RR$, since the vectors $ \lambda_{i_{n+1}}, \ldots, \lambda_{i_{n+k}}$ determines the orbit of a point in $\mathfrak{q}^{-1}(v)$ which is a $k$-dimensional torus diffeomorphic to $0 \times T^k \subset \CC^n \times T^k$. Therefore, the locally $k$-standardness of $M$ implies that the set $\{\lambda_{i_1}, \ldots, \lambda_{i_n}\}$ is a part of a $\ZZ$-basis in $\ZZ^{n+k}$.

%
%Note that each vertex $v$ of $P$ is the transversal intersection
%of $n$ many facets $\{F_{i_1}, \cdots, F_{i_n}\}$. Therefore the manifolds
%$M_{i_1}, \ldots, M_{i_n}$ intersects transversely by the locally standardness.
%Then torus $T_{i_1} \times \cdots \times T_{i_n}$ is isomorphic to
%$n$-dimensional subtorus of $T^{n+1}$. So the submodule $A$ generated by
% $\{\lambda_{i_1}, \ldots, \lambda_{i_n}\}$ is $n$-dimensional. Now, take a primitive vector $\lambda_{i_{n+1}}$ such that the rank of 
% $A \oplus < \lambda_{i_{n+1}}>$ is $n+1$, hence the volume determine by
% $\{\lambda_{i_1}, \ldots, \lambda_{i_n}, \lambda_{i_{n+1}}\}$ in $\RR^{n+1}$
% is equal to the volume determine by $\{\lambda_{i_1}, \ldots, \lambda_{i_n}\}$ 
% in $A \otimes_{\ZZ} \RR$. Since $M$ is a manifold which satisfies the locally
% standard condition, the set $\{\lambda_{i_1}, \ldots, \lambda_{i_n}\}$  
%is a part of a $\ZZ$-basis in $\ZZ^{n+1}$. 

\subsection{Constructive definition}\label{def}
In this subsection, we introduce a {\it generalized hyper characteristic function} $\xi$ on a simple 
polytope $P$ of dimension $n$ and discuss the construction of a $T^{(n+k)}$-manifold associated with $P$ and $\xi$.
We denote by $\mathcal{F}(P)$ the set of facets and $V(P)$ the set of vertices of $P$.

\begin{definition}\label{def_hyper_char}
A function $\xi \colon \mathcal{F}(P) \to \ZZ^{n+k} $ is called a \emph{generalized hyper characteristic}
function if $\xi$ satisfies the following:
\begin{equation}\label{eq_cond_for_char_pair}\tag{$\star$}
\{\xi_{j_1}, \dots, \xi_{j_n}\} \text{ is a part of a $\ZZ$-basis of } \ZZ^{n+k} \text{ whenever }F_{j_1}\cap\cdots\cap F_{j_n}\neq \emptyset.
\end{equation}
 where $ \xi_j \colonequals \xi(F_j)$ for $j=1, \ldots, m$.
\end{definition}
% We consider a function 
%\begin{equation}\label{eq_gen_hyper_char}
%\xi \colon \mathcal{F}(P) \to \ZZ^{n+k}
%\end{equation}
%which assigns an integral vector on each facet. For each $v\in V(P)$, we denote by $M(v)$ the module generated by $\xi(F_{i_1}), \dots, \xi(F_{i_n})$ if $v=F_{i_1} \cap \cdots \cap F_{i_n}$. For simplicity, we often write $\xi_j \colonequals \xi(F_j)$.  

%\begin{definition}\label{def_hyper_char}
%A function $\xi$ as in \eqref{eq_gen_hyper_char} is called a \emph{generalized hyper characteristic} function if $\xi$ satisfies the following condition:
%\begin{align}\label{eq_cond_for_char_pair}\tag{$\star$}
%M(v) \text{ is a direct summand of } \ZZ^{n+k} \text{ for each } v\in V(P).
%\begin{split}
%&\text{the module generated by } \{\xi(F_{j_1}), \dots, \xi(F_{j_n})\}\\
%&\text{ is a direct summand of } \ZZ^{n+k}, \text{ whenever }F_{j_1}\cap\cdots\cap F_{j_n}\neq \emptyset.
%& //
%&
%\end{split}
%\end{align}
%where $ \xi_j \colonequals \xi(F_j)$ for $j=1, \ldots, m$.
%\end{definition}
%Here a submodule $M \subseteq \ZZ^{m}$ is a direct summand if there is a submodule $M'$ of $\ZZ^m$ such that $\ZZ^m=M \oplus M'$. 

We denote by $\im (\xi)$ the module generated by $\{\xi_1, \dots, \xi_m\}$ and by ${\rm rk}(\xi)$ its rank. In this case, $n  \leq  {\rm rk}(\xi) \leq n+k $ because of the hypothesis \eqref{eq_cond_for_char_pair}.  
We remark that the function $\xi$ is called a \emph{characteristic function} for the case $k=0$ and \emph{hyper characteristic function} if $k=1$. We refer to \cite{DJ}  and \cite{SS} respectively. For simplicity, we call $\xi$ in Definition \ref{def_hyper_char} a hyper characteristic function and call $(P, \xi)$ a \emph{hyper characteristic pair}.
We note that the map $\lambda$ in \eqref{eq_axiomatic_lambda} is a hyper characteristic function. As it is constructed from a locally $k$-standard $T$-manifold, we prefer to distinguish $\lambda$ and $\xi$  until Corollary \ref{cor_two_def_are_equiv}.

%\begin{enumerate}
%\item $P$:  $n$-dimensional simple convex polytope with $m$ facets, 
%\item $\mathcal{F}(P)\colonequals \{F_1, \dots, F_m\}$, the set of facets.
%\item \blue{ $\xi \colon \mathcal{F}(P) \to \ZZ^{n+1} $ (see the blue comment below)} such that $\xi_{j}\colonequals \xi(F_j)$ is primitive and 
%\begin{equation}\label{eq_cond_for_char_pair}\tag{$\star$}
%\{\xi_{j_1}, \dots, \xi_{j_n}\} \text{ is a direct summand of } \ZZ^{n+1} \text{ whenever }F_{j_1}\cap\cdots\cap F_{j_n}\neq \emptyset.
%\end{equation}
%\item $T\colonequals T^{n+1}$

%$\{\xi_{j_1}, \dots, \xi_{j_n}\}$ is a direct summand of $\ZZ^{n+1}$ whenever $F_{j_1} \cap \cdots \cap F_{j_n} \neq \emptyset$,

Now, we construct a $(2n+k)$-dimensional manifold with $T^{n+k}$-action as follows.
 For a point $x \in P$, let $F_{j_1} \cap \cdots \cap F_{j_\ell}$ is the face of $P$ containing $x$ in its relative interior. Then, we denote by 
 $T_x$ the subgroup of $T^{n+k}$ determined by $\{\xi_{j_1}, \dots,  \xi_{j_\ell}\}$. If $x$ belongs to the relative interior of $P$, we define $T_x$ to be the identity in $T^{n+k}$. We consider the following identification space
 \begin{equation}\label{eq_constr_DJ}
M(P, \xi)\colonequals  (T^{n+k}\times P)/_\sim,
\end{equation}
where 
\begin{equation}\label{eq_DJ_eq_rel}
(t, p)\sim (s,q) ~~ \mbox{if and only if} ~~ p=q ~~ \mbox{and} ~~ t^{-1}s \in T_p. 
\end{equation}
Here, $T^{n+k}$ acts on $M(P, \xi)$ induced by the multiplication on the first factor
of $T^{n+k} \times P$. 

\begin{proposition}\label{prop_two_definitions_are_equiv}
Let $(P, \xi)$ be a hyper characteristic pair. Then the space $M(P, \xi)$ as in
\eqref{eq_constr_DJ} is locally isomorphic to $\mathbb{C}^n \times T^k$. 
\end{proposition}
\begin{proof}
Let $\mathfrak{q} \colon M \to P $ be the orbit map.
% and $\{v_1, \ldots, v_r\}$ the set of vertices of $P$.
Let $U_v$ be the open subset of $P$ obtained by deleting
all faces of $P$ not containing the vertex $v \in V(P)$. 
Let $v = F_{i_1} \cap \cdots \cap F_{i_n}$ for a unique collection of facets
$\{F_{i_1}, \ldots, F_{i_n} \}$ of $P$. Since $P$ is a simple
polytope, there is a homeomorphism 
$$\mathfrak{f} \colon \RR^n_{\geq}  \to U_v$$  
as manifold with corners such that the facet $\{x_\ell = 0\}$ of $\RR^{n}_\geq$ maps onto $F_{i_\ell}$
for $\ell = 1, \ldots, n$. Then the set $\{\xi_{i_1}, \ldots,
\xi_{i_n}\}$ is a part of a basis by the definition of $\xi$. 

Now we choose $\zeta_1, \ldots, \zeta_k \in \ZZ^{n+k}$
such that  $\{\xi_{i_1}, \ldots, \xi_{i_n}, \zeta_1, \ldots, \zeta_k\}$ is a $\ZZ$-basis of 
$\ZZ^{n+k}$. Then, we get a diffeomorphism $\theta \colon T^{n+k} \to T^{n+k}$
determined by the linear map sending $e_\ell$ to $\xi_{i_\ell}$ for $\ell = 1, \ldots, n$ and $e_{n+r}$ to $\zeta_r$ for $r= 1, \ldots, k$, where
$\{ e_1, \ldots, e_n, e_{n+1}, \ldots, e_{n+k}\}$ is the standard basis of $\ZZ^{n+k}$. Therefore
we get the following commutative diagram
$$
\begin{tikzcd}
T^{n+k} \times \RR^n_{\geq} \arrow{r}{\theta \times \mathfrak{f}} \arrow{d}& T^{n+k} \times U_v \arrow{d}\\
(T^{n+k} \times \RR^n_{\geq})/_{\sim_s} \arrow{r}{\hat{\mathfrak{f}}}& (T^{n+k} \times U_v)/_\sim,
\end{tikzcd}
$$
where $\sim_s$ is similarly defined as the relation $\sim$ in \eqref{eq_DJ_eq_rel} using $\{e_1, \dots, e_{n}\}$ as a hyper characteristic function on facets of $\RR^{n}_\geq$.  So the map $\hat{\mathfrak{f}}$ is a homeomorphism. The space $(T^{n+k} \times \RR^n_{\geq})/_{\sim_s}$ is homeomorphic to $ \CC^n \times T^k$. Therefore, $M(P, \xi)$ is a topological manifold locally isomorphic to $\mathbb{C}^n \times T^k$, since $P = \bigcup_{v \in V(P)} U_v$. 
\end{proof}

Note that the function $\lambda$ defined in \eqref{eq_axiomatic_lambda} satisfies the condition \eqref{eq_cond_for_char_pair} of a hyper characteristic function. Proposition \ref{prop_two_definitions_are_equiv} together with similar arguments in the proofs of \cite[Lemma 1.4, Proposition 1.8]{DJ} gives the following corollary. 

\begin{corollary}\label{cor_two_def_are_equiv}
Let $M$ be a  locally $k$-standard $T$-manifold with $\dim M= 2n+k$ such that $M/T^{n+k}$ is a simple polytope $P$ and $\lambda$ is a function as defined in \eqref{eq_axiomatic_lambda}. Then, $M$ is equivariantly homeomorphic to $M(P, \lambda)$ as defined in \eqref{eq_constr_DJ}. 
\end{corollary}

We remark that   $M$ induces a smooth structure on  $M(P, \lambda)$ due to Corollary \ref{cor_two_def_are_equiv}.   Therefore,  $M(P, \lambda)$  is diffeomorphic to a locally $k$-standard $T$-manifold.  We also note that the orientation of $M(P, \lambda)$ can be induced from the orientation of $P$ and $T^{n+k}$.

\begin{lemma}\label{lem_H^1_finite_group}
Let $M$ be a  locally $k$-standard $T$-manifold with a hyper characteristic function $\lambda$ as defined in \eqref{eq_axiomatic_lambda}.  If ${\rm rk}(\lambda)= n+k$, then the fundamental group $\pi_1(M)$  is a finite abelian group.
\end{lemma}
\begin{proof}
Let $\mathfrak{q} \colon T^{n+k} \times P \to M $ be the quotient map for the equivalence relation $\sim$ in \eqref{eq_DJ_eq_rel}. Then it follows that $\mathfrak{q}^{-1}(x)$ is connected for all $x \in M $. Also we have that  $ T^{n+k} \times P $ is locally path-connected and $M$ is semi-locally simply connected since it is locally $\CC^n \times T^k$.
Then \cite[Theorem 1.1]{CGM} gives a surjective map 
$$\pi_1(\mathfrak{q}) \colon \pi_1(T^{n+1} \times P) \twoheadrightarrow \pi_1(M).$$ 
Since $P$ is contractible, $ \pi_1 (T^{n+k} \times P) = \pi_1(T^{n+k}).$

Let $S^1(\lambda_i)$ be the circle subgroup of $T^{n+k}$ determined by $\lambda_i$ for $i = 1, \ldots, m$. So each $S^1(\lambda_i)$ is a loop in $T^{n+k}$ containing the identity. Let $\omega_i \in \pi_1(T^{n+k})$ represent this loop for $i=1, \ldots, m$. Then, $\omega_i = \lambda_{i,1}e_1 +  \cdots + \lambda_{i,n+k}e_{n+k}$ with respect to the standard generators $\{e_1, \ldots, e_{n+k}\}$ of $\pi_1(T^{n+k})$, where we denote $\lambda_i \colonequals (\lambda_{i,1}, \ldots, \lambda_{i,n+k})\in \ZZ^{n+k}$.  Since $\rk (\lambda)=n+k$, the quotient  
$\pi_1(T^{n+k})/ \left<\omega_1, \ldots, \omega_m\right>$ is a finite abelian group. Under the quotient map $\mathfrak{q}$, the circle $S^1(\lambda_i)$ collapses to a point in $M$. So  $\pi_1(\mathfrak{q})(\omega_i)$ is the identity in $\pi_1(M)$ for $i=1, \dots, m$. Hence, the lemma follows.
\end{proof}
Above lemma extends \cite[Theorem 1.1]{Ler2} which is about the fundamental group of contact toric manifolds of Reeb type. 

\begin{proposition}\label{prop_rank_n_n+1_classification}
Let $P$ and $\xi$ be as above. Then, $\rk (\xi)=n+r$ for some $0\leq r <k$ if and only if $M(P,\xi)$ is homeomorphic to $N(P, \tilde{\xi})\times T^{k-r}$ for some locally $r$-standard  $T$-manifold $N(P, \tilde{\xi})$ of dimension $2n+r$  with $\rk(\tilde{\xi})=n+r$. 
\end{proposition}
\begin{proof}
Assume that  $\rk(\xi)=n+r$ for some $0\leq r <k$. Let 
$$M(\xi)\colonequals (\langle \xi_1, \dots, \xi_m \rangle \otimes_\ZZ \RR) \cap \ZZ^{n+k}.$$
Then, $M(\xi)$ is a direct summand of $\ZZ^{n+k}$, hence we have $\ZZ^{n+k}\cong M(\xi)\oplus M(\xi)^\perp$. This induces a decomposition 
$T^{n+k}\cong T_{\xi}^{n+r} \times T_{\xi}^{\perp},$ where $T_{\xi}^{n+r}$ and $T_{\xi}^{\perp}$ are tori of dimension $n+r$ and $k-r$, respectively. 

Now, we have 
$$M(P, \xi)=(T\times P )/_\sim \cong ((T_{\xi}^{n+r} \times T_{\xi}^{\perp})\times P) / _\sim = 
(T_{\xi}^{n+r}\times P) / _\sim \times  T_{\xi}^{\perp},$$
where the last equality follows because the equivalent relation $\sim$ does not identify $T_{\xi}^{\perp}$. Here, we define $\tilde{\xi}$ by the composition 
$$
\begin{tikzcd}
\mathcal{F}(P) \arrow{r}{\xi} \arrow{rrd}[swap]{\tilde{\xi}}& \ZZ^{n+k} \arrow{r}{\cong} &M(\xi) \oplus M(\xi)^\perp \arrow{d}{pr_1}\\
&&M(\xi).
\end{tikzcd}
$$
This establishes ``only if'' part of the statement. 

Conversely, assume that $M(P, \xi)$ is equivariantly homeomorphic to $ T^{k-r}\times N(P, \tilde{\xi})$ for some locally $r$-standard  $T$-manifold $N(P, \tilde{\xi})$ with $\rk(\tilde\xi)=n+r$.  Observe that
\begin{equation}\label{eq_fund_gp}
\pi_1( T^{k-r}\times N(P, \tilde{\xi}))\cong \ZZ^{k-r} \times G
\end{equation}
for $G=\pi_1(N(P, \tilde{\xi}))$ which is a finite abelian group by Lemma \ref{lem_H^1_finite_group}. Therefore, if $\rk(\xi)\neq n+r$, then the proof of ``only if'' part and Lemma \ref{lem_H^1_finite_group} contradicts \eqref{eq_fund_gp}. Hence, we conclude $\rk(\xi)=n+r$. 
\end{proof}

Because of Proposition \ref{prop_rank_n_n+1_classification}, it suffices to assume $k \in \{0, \ldots, m-n\}$, since the maximum of $\rk(\xi)$ is $m$, namely the number of facets of $P$. In particular, if the hyper characteristic function $\xi$ is given by $\xi(F_i)=e_i$ which is the $i$-th standard unit vector in $\ZZ^m$, the resulting space $M(P, \xi)$ is called the \emph{moment-angle manifold} which we discuss below. 

%Independently from \eqref{eq_constr_DJ}, we introduce another construction of a $(2n+k)$-dimensional manifold with $T^{n+k}$-action out of a hyper characteristic pair $(P, \xi)$. Here, we further assume that $\rk (\xi)=n+k \geq n+1$. The case of rank $n$ is discussed in Proposition \ref{prop_rank_n_n+1_classification}. 

Let $x$ be a point in the relative interior of $F_{j_1} \cap \cdots \cap F_{j_\ell}$. We write 
%denote by $\widetilde{T}_x$ the subgroup of $T^m$ generated by $j_1, \dots, j_\ell$-th coordinate circles. 
$$\widetilde{T}_x \colonequals \{(t_1, \dots, t_m)\in T^m \mid t_i = 1 ~\mbox{for} ~ i \notin \{j_1, \dots, j_\ell\}\}.$$
When $x$ belongs to the relative interior of $P$, we define $\widetilde{T}_x$ to be the identity in $T^m$.  One can define the identification space 
\begin{equation*}\label{eq_MAC}
\mathcal{Z}_P \colonequals (T^m \times P) /_{\sim_z}
\end{equation*}
which is called the \emph{moment-angle manifold} associated to $P$, see for instance \cite[Chapter 6]{BP-book}. Here, $(t, p)\sim_{z} (s,q)$ if and only if $p=q$ and $t^{-1}s \in \widetilde{T}_p$. Notice that $\mathcal{Z}_P$ is equipped with a $T^{m}$-action given by the coordinate multiplication on the first factor of  $T^m \times P$. This makes  $\mathcal{Z}_P$ a locally $(m-n)$-standard $T$-manifold. 

Assume that $\rk(\xi)=n+k$. Then, regarding $\xi$ as a matrix of size $(n+k)\times m$ by listing $\xi_1, \dots, \xi_m$  as its column vectors, we have a short exact sequence 
$$\begin{tikzcd} 1 \arrow{r} & \ker (\exp \xi) \arrow{r}& T^m \arrow{r}{\exp{\xi}}
& T^{n+k} \arrow{r} & 1.
\end{tikzcd}$$
Now, we consider the space 
\begin{equation}\label{eq_const_quotient}
X(P, \xi) \colonequals \mathcal{Z}_P/ \ker (\exp \xi),
\end{equation} 
where the action of $\ker (\exp \xi)$ factors through $T^m$. The torus $T^{n+k}
\cong T^m/\ker (\exp \xi)$ acts on $X(P, \xi)$ residually. Note that the condition $(\star)$ in Definition \ref{def_hyper_char} implies that $\ker(\exp\xi) \cap T^m_p =\{id\}$ 
for  every $p \in P$. Hence,  $\ker(\exp\xi)$ acts freely on $\mathcal{Z}_P$.

The proof of following proposition is same as the standard argument in toric topology, for instance see \cite[Proposition 7.2.1]{BP-book}.
% and justifies the name of manifolds in Definition \ref{def_axiom}. 

\begin{proposition}\label{prop_axion=const}
Let $(P, \xi)$ be a hyper characteristic pair with $\rk (\xi)=n+k$. Then there is a weakly equivariant homeomorphism  between $M(P, \xi)$ in \eqref{eq_constr_DJ} and $X(P, \xi)$ in \eqref{eq_const_quotient}. 
\end{proposition}
Here the “weakness” of equivariant homeomorphism arises by the existence of 
automorphism between the standard torus $T^{n+k}$ with residual torus $T^m/\ker(\exp \xi)$.

We now exhibit several classes of examples of  locally $k$-standard $T$-manifolds.
When $k=0$ in Definition \ref{def:axiom_top_cont}, the resulting category of manifolds is introduced in \cite{DJ} which are called quasitoric manifolds. The case where  $k  = |\mathcal{F}(P)| - \dim P=m-n$ contains all moment angle manifolds associated to  simple polytopes as we discussed above.  When $k=1$, the category of $M(P, \xi)$ contains \emph{good contact toric manifolds} introduced in \cite{Lerman}. See Example \ref{ex_good_contact}.

%\begin{example}[Generalized lens spaces]
%Let $P$ be the $n$-dimensional simplex $\Delta^n$, and $\xi$ is a hyper characteristic function on it. We further assume that $\rk (\xi)=n+1$, see Proposition \ref{prop_rank_n_n+1_classification}. In this case, $\mathcal{Z}_{\Delta^{n}} \cong S^{2n+1}$ and $\ker(\exp \xi)$ is isomorphic to a finite abelian group. The quotient space $S^{2n+1}/ \ker(\exp \xi)$ is called a \emph{generalized lens space} in \cite{SS}.  
%In particular, if $\{\xi(F) \mid F\in \mathcal{F}(\Delta^n)\}$ form a basis of $\ZZ^{n+1}$, then $M(\Delta^n, \xi)$ is homeomorphic to $S^{2n+1}$. 
%\end{example}

\begin{example}[Good contact toric manifolds]\label{ex_good_contact}
Let $P$ be an $n$-dimensional simple lattice polytope embedded in $\RR^{n+1}\setminus \{\mathbf{0}\}$. Consider the cone $C(P)$ on $P$ with apex $\mathbf{0}\in \mathbb{R}^{n+1}$, and  the set $\{\tilde{F} \mid F\in \mathcal{F}(P)\}$ of facets of $C(P)$, where $\tilde{F}\colonequals C(F)\setminus \{0\}$.  Now, define a function 
$\xi\colon \mathcal{F}(P) \to \mathbb{Z}^{n+1}$
by $\xi(F)$ to be the primitive outward normal vectors of $\tilde{F}$, and assume that $\xi$ satisfies the condition \eqref{eq_cond_for_char_pair}.  Then, the resulting space $M(P, \xi)$ is $T^{n+1}$-equivariantly homeomorphic to a good contact toric manifold whose moment cone is  $C(P)$. 
%We refer to \cite[Section 2]{Ler} or \cite[Chapter 2]{Luo-T} for details of Hamiltonian torus action on a contact toric manifold and its moment map. 
\end{example}

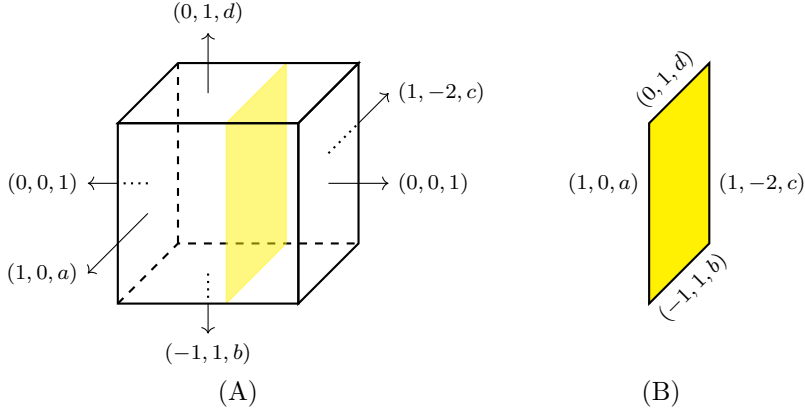
\begin{figure}
\begin{tikzpicture}[scale=0.8]
\begin{scope}
\draw[dashed, thick] (0,0)--(1,1)--(4,1); \draw[dashed, thick] (1,1)--(1,4);
\draw[fill=yellow, yellow, opacity=0.5] (1.8,0)--(2.8,1)--(2.8, 4)--(1.8,3)--cycle;
\draw[thick] (0,0)--(3,0)--(3,3)--(0,3)--cycle;
\draw[thick] (3,0)--(4,1)--(4,4)--(3,3)--cycle;
\draw[thick] (4,4)--(1,4)--(0,3)--(3,3)--cycle;

\draw[->] (1.5,3.5)--(1.5,4.5); 
\node[above] at (1.5, 4.5) {\footnotesize$(0,1,d)$};

\draw[thick, dotted] (1.5,0.5)--(1.5,0); \draw[->] (1.5, 0)--(1.5, -0.5);
\node[below] at (1.5, -0.5) {\footnotesize$(-1,1,b)$};

\draw[thick, dotted] (0.5,2)--(0,2); \draw[->] (0,2)--(-0.5, 2);
\node[left] at (-.5, 2) {\footnotesize$(0,0,1)$};

\draw[->] (3.5,2)--(4.5,2); 
\node[right] at (4.5, 2) {\footnotesize$(0,0,1)$};

\draw[->] (0.5,1.5)--(-.5,0.5); 
\node[left] at (-.5, 0.5) {\footnotesize$(1,0,a)$};

\draw[thick, dotted] (3.5,2.5)--(4,3); \draw[->] (4,3)--(4.5, 3.5);
\node[right] at (4.5, 3.5) {\footnotesize$(1,-2,c)$};

\node at (2,-1.5) {(A)};
\end{scope}

\begin{scope}[xshift=200]
\draw[fill=yellow, thick] (1.8,0)--(2.8,1)--(2.8, 4)--(1.8,3)--cycle;
\node[left] at (1.8, 2) {\footnotesize$(1,0,a)$};
\node[right] at (2.8, 2) {\footnotesize$(1,-2,c)$};
\node[above, rotate=45] at (2.3, 3.5) {\footnotesize$(0,1,d)$};
\node[below, rotate=45] at (2.3, 0.5) {\footnotesize$(-1,1,b)$};
\node at (2,-1.5) {(B) };
\end{scope}
\end{tikzpicture}
\caption{Hyperplane cut of a quasitoric manifold; (A) Characteristic pair of a quasitoric manifold, 
(B) Hyper characteristic pair of a  locally $1$-standard $T$-manifold of dimension $5$.}
\label{fig_hyperplane_cut}
\end{figure}

\begin{example}[Hyperplane cut of a quasitoric manifold] \label{ex:topocont_toric_from_toric}
Let $X$ be a $2n$-dimensional quasitoric manifold which is a locally 0-standard $T^{n}$-manifold and $\mathfrak{q}
\colon X \to Q$ be the associated orbit map. Let $H$ be a hyperplane in
$\RR^n$ which does not contain any vertex of $Q$. Since $Q$
 is an $n$-dimensional simple polytope, $P \colonequals Q \cap H$ is an $(n-1)$-dimensional
 simple polytope. Then $\mathfrak{q}^{-1}(P)$ is a $T^n$-invariant subspace
 of $X$. 
 
Note that if $v$ is a vertex of $P$, then $v$ is the intersection of an edge $e$ and $H$. Let $P_v$ be the open subset of $P$ obtained by deleting all faces of $P$ not containing $v$ and $Q_v$ be the open subset of $Q$ obtained by deleting  all faces of $Q$ not containing the edge $e$. Then $Q_v$ is homeomorphic to $P_v \times \mathring{e}$ as manifold with
 corners where $\mathring{e}$ is the relative interior of $e$.
So $\mathfrak{q}^{-1}(Q_v) = \mathfrak{q}^{-1}(P_v) \times \mathring{e}$.
Since $P_v$ and $\mathring{e}$ intersect transversally,
$\mathfrak{q}^{-1}(P_v)$ is a codimension-1 submanifold of  $\mathfrak{q}^{-1}(Q_v)$.
 Therefore, $\mathfrak{q}^{-1}(P)$ is a $(2n-1)$-dimensional manifold
 with  an effective $T^n$-action which satisfies the condition of Definition 
 \ref{def:axiom_top_cont}. Hence $\mathfrak{q}^{-1}(P)$ is a  locally $1$-standard $T$-manifold of dimension $2n-1$. We observe this in the following particular case.

Consider $Q$ and $P$ as described in Figure \ref{fig_hyperplane_cut}-(A)
and (B) respectively. The space $\mathfrak{q}^{-1}(P)$ is
a $S^1$-bundle over $\CP^2 \# \CP^2$.
We recall that good contact toric manifolds bijectively correspond to the moment cones
see \cite{Lerman}. Here, hyper characteristic vectors on $P$ do not satisfy the condition
of a moment cone. So $\mathfrak{q}^{-1}(P)$ is not a contact toric manifold, but in the category of the case where $k=1$. 
\end{example}

A similar procedure as in Example \ref{ex:topocont_toric_from_toric} applies to a  locally $k$-standard $T$-manifold,  which provides a  new locally $(k+1)$-standard $T$-manifold which is of codimension 1 in the original manifold.

\section{Equivariant cohomology of  locally $k$-standard $T$-manifolds}\label{sec_equiv_coh_of_tctm}
In this section, we study the equivariant cohomology algebra of a locally $k$-standard $T$-manifold $M(P,\xi)$. For simplicity, we write $M \colonequals M(P, \xi)$ and $T\colonequals T^{n+k}$ throughout the remaining part of this paper. 

\subsection{$H_T^\ast(M)$ as a ring}
Consider the set $V(P)\colonequals \{v_1, \dots, v_\ell\}$ of vertices of $P$ and write $v_i:=F_{i_1} \cap \dots \cap F_{i_n}$, the intersection of $n$ facets $F_{i_1}, \dots, F_{i_n}$. We denote by $T_{v_{i}}$ the subtorus of $T$ generated by $\{\xi_{i_1}, \dots, \xi_{i_n}\} \subset \ZZ^{n+k}$ and write $T^\perp_{v_i} \colonequals T/T_{v_i}$ which is isomorphic to the $k$-dimensional torus $T^k$. Then, we have a short exact sequence 
\begin{equation}\label{eq_ses_for_each_vertex}
\begin{tikzcd}
1 \arrow{r} & T_{v_i} \arrow[yshift=0ex]{r}{\alpha_i} & T \arrow{r} & T^\perp_{v_i} \arrow{r} & 1
\end{tikzcd}
\end{equation}
for each $i=1, \dots, \ell$. We notice that, for each vertex $v_i\in V(P)$ the constructive definition of $M$ shows $\pi^{-1}(v_i)=T^\perp_{v_i}\times v_i$. 

Identifying $H_2(BT)$ with the lattice $\left< e_1, \dots, e_{n+k}\right>$ generated by the standard unit vectors $\{e_1,\dots, e_{n+k}\}$ of $\ZZ^{n+k}$, we have  $H_2(BT_{v_i})\cong \left<\xi_{i_1}, \dots, \xi_{i_n}\right>$, the sublattice of $\ZZ^{n+k}$ generated by $\{\xi_{i_1}, \dots, \xi_{i_n}\}$. It is a direct summand of $\ZZ^{n+k}$ by the hypothesis \eqref{eq_cond_for_char_pair}. Hence, we get 
\begin{equation*}\label{eq_homology_H(BT)_decomp}
H_2(BT) \cong H_2(BT_{v_i}) \oplus N_{v_i}
\end{equation*}
for some $k$-dimensional $\ZZ$-submodule $N_{v_i}$ of $H_2(BT)$. This gives us the following identification of cohomology groups
\begin{equation}\label{eq_cohomology_H(BT)_decomp}
H^2(BT_{v_i})\cong H^2(BT)/ \Ann (H_2(BT_{v_i})),
\end{equation}
which is also isomorphic to the annihilator $\Ann (N_{v_i})$ of $N_{v_i}$.
\begin{lemma}\label{lem_H*(BT_v)=H*(S_v)}
$H^\ast_{T}(T^\perp_{v_i}\times v_i) \cong H^\ast(BT_{v_i})$. 
\end{lemma}
\begin{proof}
Hypothesis \eqref{eq_cond_for_char_pair} gives that \eqref{eq_ses_for_each_vertex} is a split exact sequence, which yields an identification $T\cong T_{v_i} \times T^\perp_{v_i}$. Hence, we have 
\begin{align*}
H^\ast_{T}(T^\perp_{v_i}\times v_i) &=H^\ast(ET \times_{T} T^\perp_{v_i}) \\
&=H^\ast(ET_{v_i}\times_{T_{v_i}}(ET^\perp_{v_i} \times_{T^\perp_{v_i}}T^\perp_{v_i})\\
&=H^\ast(BT_{v_i}),
\end{align*}
where  the third equality holds, because $ET^\perp_{v_i}$ is contractible and $T^\perp_{v_i}$ acts freely on $T^\perp_{v_i}$. 
\end{proof}

Recall from Section \ref{sec_tctm} that each facet $F_j\in \mathcal{F}(P)$ is associated with a codimension $2$ submanifold $M_j$ which is fixed by the circle subgroup of $T$ generated by $\xi_j\in \ZZ^{n+k}\cong \Hom (S^1, T)$. Let $\tau_j\in H^2_T(M)$ be the equivariant Thom class of the normal bundle $\nu(M_j)$ of $M_j$ in $M$. To be more precise, it is the image of the identity in $H^{\ast}_{T}(M_j)$ via the equivariant Gysin homomorphism $H_{T}^\ast(M_j) \to H_{T}^{\ast+2}(M)$. 

Now, we consider the map 
\begin{equation}\label{eq_f_v_i}
f_{v_i}\colon H^\ast_T(M) \to H^\ast_T(T^\perp_{v_i})
\end{equation} 
induced from the inclusion $T^\perp_{v_i} \times v_i \hookrightarrow M$. Here we denote by $\tau_j|_{v_i}$ the image of $\tau_j$ via $f_{v_i}$. It satisfies the following property:

\begin{lemma}\label{lem_res_thom_class_first_lem}
For a vertex $v_i=F_{i_1}\cap \dots \cap F_{i_n}$ of $P$, 
\begin{enumerate}
\item  $\tau_j|_{v_i}=0$ if $j\notin \{i_1, \dots, i_n\}$; 
\item $\{\tau_{i_1}|_{v_i}, \ldots, \tau_{i_n}|_{v_i}\}$ is a linearly independent set. 
\end{enumerate}
\end{lemma}
\begin{proof}
For $j\notin \{i_1, \dots, i_n\}$, two subspaces $\pi^{-1}(F_j)=M_j$ and $\pi^{-1}(v_i)=T^\perp_{v_i}\times v_i$ do not intersect. Hence, the assertion (1) follows. 

To show  the second assertion, let $\nu(T^\perp_{v_i})$ be the normal bundle of $T^\perp_{v_i}\times v_i$ in $M$ and consider its $T$-equivariant decomposition 
\begin{equation}\label{eq_normal_bundle_decomp}
\nu(T^\perp_{v_i}) \cong \bigoplus_{k=1}^n \nu(M_{i_k})|_{T^\perp_{v_i}}. 
\end{equation}
Here, $\nu(M_{i_k})|_{T^\perp_{v_1}}$ denotes the restriction of $\nu(M_{i_k})$ to $T^\perp_{v_i}\times v_i$. The equivariant Euler class of $\nu(M_{i_k})|_{T^\perp_{v_i}}$ agrees with $\tau_{i_k}|_{v_i}\in H^2_T(T^\perp_{v_i})$ for each $k=1, \dots, n$ by the naturality. Hence, the claim follows from the equivariant decomposition in \eqref{eq_normal_bundle_decomp}. 
\end{proof}

We now discuss more properties on equivariant Thom classes for the preparation of the study in Section \ref{sec_equiv_cohom_rigidity}. 
Let  $\mathbf{x}_{j,i}\in H^2(BT)$ be a representative of $\tau_j|_{v_i}\in H^2_T(T^\perp_{v_i})\cong H^2(BT)/\Ann (H_2(BT_{v_i}))$ via \eqref{eq_cohomology_H(BT)_decomp} and Lemma \ref{lem_H*(BT_v)=H*(S_v)}. Furthermore,  we may take $\mathbf{x}_{j,i}=0$ for $j\notin \{i_1, \dots, i_n\}$ by Lemma \ref{lem_res_thom_class_first_lem}. Hence, we write $\mathbf{x}_{i_k}:=\mathbf{x}_{i_k, i}$ for simplicity. Identifying $H_2(BT)$ and $H_2(BT_{v_i})$ with $\ZZ^{n+k}$ and $\left< \xi_{i_1}, \dots, \xi_{i_n}\right>\subset \ZZ^{n+k}$ respectively, we regard $\mathbf{x}_{i_1}, \dots, \mathbf{x}_{i_n}$ as elements of $(\ZZ^{n+k})^\ast$.

\begin{lemma}\label{lem_rep_of_Thom_class}
For each vertex $v_i=F_{i_1}\cap \dots \cap F_{i_n}$, any representative $\mathbf{x}_{i_s}\in (\ZZ^{n+k})^\ast$ of $\tau_{i_s}|_{v_i}$ is an element in 
$$\Ann \left< \xi_{i_1}, \dots, \xi_{i_{s-1}},\xi_{i_{s+1}}, \dots,  \xi_{i_n}\right> \setminus \Ann \left<\xi_{i_1}, \dots, \xi_{i_n}\right>,$$
%the set $\{\mathbf{x}_{i_1, i}, \dots, \mathbf{x}_{i_n, i}\} \subset H^2(BT)$ is linearly independent. 
i.e., $\left<  \xi_{i_r}, \mathbf{x}_{i_s}\right>=0$ for all $r\in \{1, \dots, n\}\setminus \{s\}$ and $\left< \xi_{i_s}, \mathbf{x}_{i_s}\right>\neq 0$. 
\end{lemma}
\begin{proof}
Recall the following identifications;
\begin{enumerate}
\item $H_2(BT)\cong \ZZ^{n+k} \cong \Hom(S^1, T)$. We denote by $\lambda_\xi \in \Hom(S^1, T)$ the weight corresponding to $\xi\in \ZZ^{n+k} \cong H_2(BT)$. 
\item $H^2(BT)\cong (\ZZ^{n+k})^\ast \cong \Hom(T, S^1)$. We denote by $\chi^{\zeta}\in \Hom(T, S^1)$ the character corresponding to $\zeta \in (\ZZ^{n+k})^\ast \cong H^2(BT)$. 
\item For each $t\in S^1$, we have 
\begin{equation}\label{eq_char_weight}
(\chi^{\zeta}\circ \lambda_\xi)(t)=t^{\left<  \xi, \zeta \right>}, 
\end{equation}
where $\left<~,~ \right>$ denotes the standard paring between elements in $\ZZ^{n+k}$ and its dual $(\ZZ^{n+k})^\ast$. 
\end{enumerate}
Now, the equivariant decomposition of $\nu(T^\perp_{v_i})$ as in \eqref{eq_normal_bundle_decomp} together with \eqref{eq_char_weight} implies that 
$\left<  \xi_{i_r}, \mathbf{x}_{i_s}\right>=0$ for all $r\in \{1, \dots, n\}\setminus \{s\}$ and $\left< \xi_{i_s},  \mathbf{x}_{i_s}\right>\neq 0$. 
%Therefore, the representative $\mathbf{x}_{i_s}$ does not belong to $\Ann \left< \xi_{i_1}, \dots, \xi_{i_n}\right>$. 
\end{proof}

In addition to Lemma \ref{lem_rep_of_Thom_class}, one can choose a particular representative $\mathbf{x}_{i_s}\in (\ZZ^{n+k})^\ast$ of $\tau_{i_s}|_{v_i}$ such that $\left<  \xi_{i_s}, \mathbf{x}_{i_s}\right>=1.$
Indeed, we extend hyper characteristic vectors $\{\xi_{i_1}, \dots, \xi_{i_n}\}$ around a vertex $v_i$ to a basis $\{\xi_{i_1}, \dots, \xi_{i_n}, \eta_1, \dots, \eta_k \}$ of $\ZZ^{n+k}$. Then, we take first $n$ elements of $\{\mathbf{x}_{i_1}, \dots, \mathbf{x}_{i_n},\mathbf{x}_{i_{n+1}}, \dots,  \mathbf{x}_{i_{n+k}}\}\subset (\ZZ^{n+k})^\ast$ which is dual to $\{\xi_{i_1}, \dots, \xi_{i_n}, \eta_1, \dots, \eta_k \}$.  With this observation, we get the following conclusion. 

\begin{corollary} \label{cor_rep_of_thom_classes_equation}
For each vertex $v_i=F_{i_1} \cap \dots \cap F_{i_n}$ and associated set 
\begin{equation}\label{eq_set_of_thom_cl_assoc_v_i}
\{\tau_{i_1}|_{v_i}, \dots, \tau_{i_n}|_{v_i} \} \subset H^2_T(T^\perp_{v_i})\cong (\ZZ^{n+k})^\ast /\Ann\left<\xi_{i_1}, \dots, \xi_{i_n}\right>
\end{equation}
of restrictions of equivariant Thom classes, there is a set $\{\mathbf{x}_{i_1}, \dots, \mathbf{x}_{i_n}\}\subset (\ZZ^{n+k})^\ast$ of representatives of \eqref{eq_set_of_thom_cl_assoc_v_i} such that 
\begin{equation}\label{eq_rep_of_thom_classes_equation}
\left<  \xi_{i_r}, \mathbf{x}_{i_s}\right>=\begin{cases} 
0 & \text{if } r\neq s ;\\ %j \in \{i_1, \dots, i_n\}\setminus \{i_k\}; \\
1 & \text{if } r=s. %j=i_k. 
\end{cases}
\end{equation}
In particular, the submodule generated by such a set $\{\mathbf{x}_{i_1}, \dots, \mathbf{x}_{i_n}\}$ is a direct summand of $(\ZZ^{n+k})^\ast$. 
\end{corollary}

We finish this subsection by showing the following theorem about the ring structure of the equivariant cohomology of $M$.  We recall that the equivariant cohomology $H^\ast_{T^m}(\mathcal{Z}_P)$ is isomorphic to 
\begin{equation}\label{eq_face_ring}
{\rm SR}(P)\colonequals \ZZ[y_1, \dots, y_m]/\left< y_{i_1}\cdots y_{i_r} \mid F_{i_1}\cap \cdots \cap F_{i_r}=\emptyset \right>,
\end{equation}
which is called the \emph{Stanley--Reisner ring} of $P$, see \cite[Section 4]{DJ}. Here, $y_i$'s are indeterminates of degree 2.

\begin{theorem}\label{thm_equiv=SR}
Let $M\colonequals M(P, \xi)$ be a  locally $k$-standard $T$-manifold of dimension $2n+k$ such that $\im(\xi)=\ZZ^{n+k}$. Let $\tau_j$  be the equivariant Thom class of $\nu(M_j)$ for $j=1, \dots, m$. Then the equivariant cohomology  $H^*_{T}(M)$ is isomorphic to 
\begin{equation}\label{eq_SR_with_thom_class}
\mathbb{Z}[\tau_1, \dots, \tau_m]/\left< \tau_{j_1}\cdots \tau_{j_r} \mid F_{j_1}\cap \cdots \cap F_{j_r}=\emptyset \right> 
\end{equation}
as rings. 
\end{theorem}
\begin{proof}
A hyper characteristic function $\xi$ of rank $n+k$, regarded as a matrix of size $(n+k) \times m$ whose column vectors are indexed by facets of $P$, yields a short exact sequence 
\begin{equation*}\label{eq_ses_lattice}
\begin{tikzcd}
0 \arrow{r} & \ker \xi \arrow{r} & \ZZ^m \arrow{r}{\xi} & \ZZ^{n+k} \arrow{r} & 0.
\end{tikzcd}
\end{equation*}
Since $\ZZ^{n+k}$ is a free $\ZZ$-module, the above short exact sequence splits. This also implies that the following short exact sequence of tori 
%First, we assert that $\ker \xi$ is a direct summand of $\ZZ^m$. Indeed, consider the $\ZZ$-module $M\colonequals (\ker\xi \otimes \RR)\cap \ZZ^m$, which is a direct summand of $\ZZ^m$ and corresponding decomposition $\ZZ^m\cong M \oplus M^\perp$. This gives us 
%$$\ZZ^{n+k}\cong \ZZ^m / \ker\xi \cong (M\oplus M^\perp)/\ker \xi \cong M/\ker \xi \oplus M^\perp.$$
%Since both $\ZZ^{n+k}$ and $M^\perp$ are free $\ZZ$-modules, so is $M/\ker \xi$, which implies that $M=\ker \xi$ because both have the same rank. 
$$
\begin{tikzcd}
1 \arrow{r} & K \arrow{r}{\iota} & T^m \arrow{r}{\exp \xi} & T\arrow{r}& 1
\end{tikzcd}
$$
splits, where $K\colonequals \ker (\exp \xi)$ which is also equal to $\exp (\ker \xi)$ because $\im(\xi)=\ZZ^{n+k}$. 

Now, we consider the following identifications
\begin{align}\label{eq_Borel_ZP_Borel_lsktm}
\begin{split}
ET^m\times_{T^m} \mathcal{Z}_P &\cong (EK \times ET)\times_{K\times T} \mathcal{Z}_P \\
&\simeq (EK \times ET)\times_{T} \mathcal{Z}_P/K\\
&\simeq EK \times (ET\times_{T} \mathcal{Z}_P/K)\\
&\simeq ET\times_{T} M.
\end{split}
\end{align}
  Hence, we conclude that 
\begin{equation*}\label{eq_equiv_cohom=SR}
H^\ast_T(M) \cong {\rm SR}(P).
\end{equation*}

Now, it remains to verify the relation between Thom classes $\tau_j$'s of \eqref{eq_SR_with_thom_class} and $y_j$'s of ${\rm SR}(P)$. First, note that $\{\tau_1, \dots, \tau_m\}$ is linearly independent. Indeed, 
$$f_{v_i}\left(\sum_{j=1}^m a_i \tau_i \right) = \sum_{j=1}^m a_{j} \tau_{j}|_{v_i}
= \sum_{r=1}^n a_{i_r} \tau_{i_r}|_{v_i}\in H_T^2(T^\perp_{v_i}),$$
where $f_{v_i}$ is defined in \eqref{eq_f_v_i} and the second equality follows from Lemma \ref{lem_res_thom_class_first_lem}-(1). Hence, if $\sum_{j=1}^m a_i \tau_i=0$, then $\sum_{k=1}^n a_{i_k} \tau_{i_k}|_{v_i}=0$. This implies that $a_{i_1}=\cdots =a_{i_n}=0$ by  Lemma \ref{lem_res_thom_class_first_lem}-(2). The same procedures for other vertices establish $a_1=\cdots = a_m=0$. 

Next, the cup product $\tau_{i_1}\cap \cdots \cap \tau_{i_r}$ vanishes whenever $F_{i_1} \cap \cdots \cap F_{i_r}=\emptyset$, because $\tau_{i_1}\cap \cdots \cap \tau_{i_r}$ represents the equivariant Poincare dual of the intersection $M_{i_1} \cap \dots \cap M_{i_r}$. Hence, one has an isomorphism between \eqref{eq_SR_with_thom_class} and  ${\rm SR}(P)$ by sending $\tau_j$ to $y_j$ for $j=1, \dots m$. 
%{\textcolor{red}{(Check the final paragraph carefully!!)}}. 
\end{proof}

The authors of \cite{DJ} showed that the Borel construction of a quasitoric manifold and the Borel construction of the corresponding moment angle manifold are same as in \eqref{eq_Borel_ZP_Borel_lsktm}. Then they use a certain decomposition of $\mathcal{Z}_P$ arising from the cubical decomposition of $P$ and Mayer--Vietories sequence to prove  that the Equivariant cohomology of $\mathcal{Z}_P$ is isomorphic to ${\rm SR}(P)$. 
 In addition to it we showed in Theorem \ref{thm_equiv=SR} that the generators of ${\rm SR}(P)$ can be identified with the equivariant Thom classes $\tau_j$'s which holds for a class of locally $k$-standard $T$-manifolds containing quasitoric manifolds and moment angle manifolds.

\subsection{Algebra structure}
Note that $H^\ast_T(M)$ is equipped with $H^\ast(BT)$-algebra structure induced from the Borel fibration 
\begin{equation}\label{eq_Borel_fibration}
\begin{tikzcd}
M \arrow[hook]{r} &ET\times_T M \arrow{r}{\pi}& BT.
\end{tikzcd}
\end{equation}
The following lemma allows us to see $H^2(BT)$ as a subset of $H^2_T(M)$.

\begin{lemma}
Let $M$ be a  locally $k$-standard $T$-manifold of dimension $ 2n+k$ with the hyper characteristic function $\xi$ and assume that $\rk(\xi)=n+k$. Then, $H^2(BT)$ is a subgroup of $H_T^2(M)$. Moreover, the inclusion of $H^2(BT)$ into $H^2_T(M)$ is the homomorphism $\pi^\ast \colon H^2(BT) \to H^2_T(M)$. 
\end{lemma}
\begin{proof}
The cohomology Leray--Serre spectral sequence for the fibration \eqref{eq_Borel_fibration}  gives us 
$$E^{p,q}_2 = H^p(BT;H^q(M;\ZZ)),$$
where the system of local coefficients is simple, because $BT$ is simply connected. We refer to \cite[Proposition 5.20]{McC}. Notice that $E_2^{2,0}= H^2(BT;\ZZ)$ and $E_2^{0,1}=0$ by Lemma \ref{lem_H^1_finite_group} together with Hurewicz theorem and the universal coefficient theorem, see for instance \cite[Theorem 3.4, Corollary 7.3]{Bre-TG}. Therefore, the differential 
$$d_2^{1,0} \colon E_2^{0,1} \to E_2^{2,0}$$
is a zero map, which implies that 
$$H^2(BT)=E_2^{2,0}=E_3^{2,0} = \cdots = E_\infty^{2,0} \subset H_T^2(M).$$
Now, the second assertion directly follows from \cite[Theorem 5.9]{McC}. 
\end{proof}

Hence, for each element $u\in H^2(BT)$, we have 
\begin{equation}\label{eq_algebra_str}
\pi^\ast(u)=u=\sum_{j=1}^m w_j(u)\cdot \tau_j=\sum_{j=1}^m \left<\widetilde{w}_j, u \right> \tau_j
\end{equation}
for some linear map 
$$w_j \colon H^2(BT) \to \ZZ.$$ 
Here, $\widetilde{w}_j\in \ZZ^{n+k}$ and $\left< ~, ~ \right>$ denote the element  corresponding to $w_i$ via the identifications of $\Hom (H^2(BT), \ZZ)\cong \Hom ((\ZZ^{n+k})^\ast, \ZZ) \cong \ZZ^{n+k}$ and  the standard paring between $\ZZ^{n+k}$ and $(\ZZ^{n+k})^\ast$, respectively. 

\begin{lemma}\label{lem_relation_alg_str_coeff_and_char_vec}
Let $\{\mathbf{x}_{i_1}, \dots, \mathbf{x}_{i_n}\}\subset \ZZ^{n+k}$ be the set as in Corollary \ref{cor_rep_of_thom_classes_equation}. Then, $\{\widetilde{w}_{i_1}, \dots, \widetilde{w}_{i_n}\}$ satisfies 
\begin{equation}\label{eq_tilde_w_dual_basis_relation}
\left< \widetilde{w}_{i_r},  \mathbf{x}_{i_s}\right>=\begin{cases} 
0 & \text{if } r\neq s;\\ %j \in \{i_1, \dots, i_n\}\setminus \{i_k\}; \\
1 & \text{if } r=s.%j=i_k. 
\end{cases}
\end{equation}
\end{lemma}
\begin{proof}
Consider the following commutative diagram,
\begin{equation}\label{eq_algebra_structure_deg2_linear}
\begin{tikzcd}[row sep=scriptsize]
H^2(BT)\arrow[hook]{r}{\pi^\ast} \arrow{d}{\cong} &H^2_T(M) \arrow{r}{f_{v_i}}  & H^2_T(T^\perp_{v_i}) \arrow{d}{\cong}\\
(\ZZ^{n+k})^\ast  \arrow[two heads]{rr}{pr_{v_i}}&& (\ZZ^{n+k})^\ast/\Ann\left< \xi_{i_1}, \dots, \xi_{i_n}\right>
\end{tikzcd}
\end{equation}
where $\pi^\ast$ and $f_{v_i}$ are defined in \eqref{eq_algebra_str} and \eqref{eq_f_v_i}, respectively and $pr_{v_i}$ is the projection. 
Now, we evaluate $\mathbf{x}_{i_k}\in (\ZZ^{n+k})^\ast$. Then, we have 
\begin{equation}\label{eq_algebra_computation}
f_{v_i}(\pi^\ast(\mathbf{x}_{i_s}))=f_{v_i}(\mathbf{x}_{i_s})
=f_{v_i}\left( \sum_{j=1}^m \left< \widetilde{w}_j, \mathbf{x}_{i_s} \right> \tau_j \right)
=\sum_{r=1}^n\left<  \widetilde{w}_{i_r}, \mathbf{x}_{i_s}\right> \tau_{i_r}|_{v_i}
\end{equation}
and $pr_{v_i}(\mathbf{x}_{i_s})=\tau_{i_s}|_{v_i}$. Hence the commutativity of the diagram \eqref{eq_algebra_structure_deg2_linear} shows 
$$\tau_{i_s}|_{v_i}=\sum_{r=1}^n\left<  \widetilde{w}_{i_r}, \mathbf{x}_{i_s} \right> \tau_{i_r}|_{v_i},$$
which implies the desired relation \eqref{eq_tilde_w_dual_basis_relation}, because the set $\{\tau_{i_1}|_{v_i}, \dots, \tau_{i_n}|_{v_i}\}$ is linearly independent, see Lemma \ref{lem_res_thom_class_first_lem}. 
\end{proof}

%\begin{remark}
We note that relations \eqref{eq_rep_of_thom_classes_equation} and  \eqref{eq_tilde_w_dual_basis_relation} are independent from the choice of a vertex of $P$. To be more precise, if two distinct vertices $v_i$ and $v_\ell$ are contained in a common facet $F_j$, then both $\xi_j$ in \eqref{eq_rep_of_thom_classes_equation} and $\widetilde{w}_j$ in \eqref{eq_tilde_w_dual_basis_relation} are related to two different 
set of dual elements $\mathcal{B}_1\colonequals \{\mathbf{x}_{i_1}, \dots, \mathbf{x}_{i_n}\}$ and $\mathcal{B}_2\colonequals \{\mathbf{y}_{\ell_1}, \dots, \mathbf{y}_{\ell_n}\}$ corresponding to restrictions of equivariant Thom classes around $v_i$ and $v_\ell$, respectively. Two relations  \eqref{eq_rep_of_thom_classes_equation} and  \eqref{eq_tilde_w_dual_basis_relation}  holds for both $\mathcal{B}_1$ and $\mathcal{B}_2$. 

Now, we prove the following theorem. 

\begin{theorem}\label{thm_recovering_char_vec}
Let $M$ be a  locally $k$-standard $T$-manifold of dimension $2n+k$ with $\rk(\xi)=n+k$. 
Let $\widetilde{w}_j$ be the vector as in \eqref{eq_algebra_str}. Then, 
$\widetilde{w}_j = \xi_j$ for each $j=1, \dots, m$.
\end{theorem}
\begin{proof}
We claim that 
\begin{equation}\label{eq_claim_of_main_thm_case1}
\left< \widetilde{w}_j - \xi_j, \mathbf{x} \right> =0, \text{ for all } \mathbf{x}\in (\ZZ^{n+k})^\ast. 
\end{equation}
Take a vertex $v$ of given facet $F_j$ and assume $v=F_{i_1}\cap \cdots \cap F_{i_n}$, namely, $j$ is an element of $\{i_1, \dots, i_n\}$. Then, Corollary \ref{cor_rep_of_thom_classes_equation} gives us a decomposition
$$(\ZZ^{n+k})^\ast \cong \left< \mathbf{x}_{i_1}, \dots, \mathbf{x}_{i_n} \right> \oplus \Ann \left<\xi_{i_1}, \dots, \xi_{i_n}\right>. $$ 
For elements $\mathbf{x} \in \Ann \left<\xi_{i_1}, \dots, \xi_{i_n}\right>$, we apply \eqref{eq_algebra_computation} to get
$$
0=f_{v}(\pi^\ast(\mathbf{x}))=f_{v}(\mathbf{x})
=f_{v}\left( \sum_{j=1}^m \left< \widetilde{w}_j, \mathbf{x} \right> \tau_j \right)
=\sum_{r=1}^n\left<  \widetilde{w}_{i_r}, \mathbf{x}\right> \tau_{i_r}|_{v}.
$$
Since the set $\{\tau_{i_1}|_{v}, \dots, \tau_{i_n}|_{v}\}$ is linearly independent, we have $\left<\widetilde{w}_{i_r}, \mathbf{x}\right>=0$ for all $r=1, \dots, n$. In particular, $\left< \widetilde{w}_j, \mathbf{x}\right>=0$ because $j$ is an element of $\{i_1, \dots, i_n\}$. Moreover, $\left<\xi_{j}, \mathbf{x} \right>=0$ because $\mathbf{x} \in \Ann \left<\xi_{i_1}, \dots, \xi_{i_n}\right>$. Hence, the assertion \eqref{eq_claim_of_main_thm_case1} has been established for $\mathbf{x}\in \Ann \left<\xi_{i_1}, \dots, \xi_{i_n}\right>$. For elements $\mathbf{x}\in  \left< \mathbf{x}_{i_1}, \dots, \mathbf{x}_{i_n} \right>$, the claim follows from Corollary \ref{cor_rep_of_thom_classes_equation} and Lemma \ref{lem_relation_alg_str_coeff_and_char_vec}. Hence, the result follows. 

\end{proof}

\section{Equivariant cohomological rigidity}\label{sec_equiv_cohom_rigidity}
In this section, we answer the equivariant cohomological rigidity problem for the category of  locally $k$-standard $T$-manifolds. The main theorem (Theorem \ref{thm_main_rigidity}) states that the weak isomorphism classes of the equivariant cohomology distinguishes the weak homeomorphism classes of locally $k$-standard $T$-manifolds. Here, weak homeomorphism between two $T$-manifolds $M$ and $M'$ means a homeomorphism $\Psi \colon M\to M'$ such that $\Psi(t \cdot x)= \delta(t)\cdot \Psi(x)$ for arbitrary $t\in T$ and $x\in M$ and for some automorphism $\delta\in \Aut(T)$. A weak isomorphism between two equivariant cohomology algberas are similarly defined with respect to an automorphism of $H^\ast(BT)$.

We first prepare the following proposition which tells about the combinatorics of the simple polytopes from their Stanley-Reisner rings. It can be deduced from the result of \cite{BrGu} as well as the relationship between simple convex polytope and its dual simplicial complex.   
\begin{proposition}\label{prop_brgu}
Let $P$ and $P'$ be simple polytopes with facets $ \mathcal{F}(P) $ and  $ \mathcal{F}(P')$ respectively. Suppose  ${\rm SR}(P)$ and ${\rm SR}(P')$ are isomorphic as $\ZZ$-algebras (see \eqref{eq_face_ring}). Then there is  a bijection $\phi \colon \mathcal{F}(P) \to \mathcal{F}(P')$ which induces a  face preserving homeomorphism $\bar \phi \colon P\to P'$. 
\end{proposition}

 Recall from Corollary \ref{cor_two_def_are_equiv} and Proposition \ref{prop_axion=const} that a locally $k$-standard $T$-manifold is determined by its hyper characteristic pair. Now, we introduce our main results.

\begin{theorem}\label{thm_main_rigidity}
Let $M$ and $M'$ be locally $k$-standard $T$-manifolds associated with $(P,\xi)$ and $(P, \xi')$ respectively such that $\im(\xi)$ and $\im(\xi')$ are direct summands of $\ZZ^{n+k}$. Then, $M$ and $M'$ are weakly equivariantly  homeomorphic if and only if their equivariant cohomology algebras are weakly isomorphic. 
\end{theorem}
\begin{proof} 
Suppose we have two locally $k$-standard $T$-manifolds $M$ and $M'$ of dimension $(2n+k)$. Let 
$\psi \colon H^\ast_T(M) \to H^\ast_T(M')$
be a weak isomorphism as $H^\ast(BT)$-algebras. We adhere notations discussed in Section \ref{sec_tctm}. 

\vspace{0.1cm}
\noindent\textbf{Case 1}. We first consider the case when $\xi$ and $\xi'$ are surjective. Theorem \ref{thm_equiv=SR} implies that the equivariant cohomology ring of $M$ is isomorphic to the Stanley--Reisner ring ${\rm SR}(P)$ of the orbit space $P$ of $M$. Hence, an algebra isomorphism $\psi$ induces an isomorphism between ${\rm SR}(P)$ and ${\rm SR}(P')$. Using Proposition \ref{prop_brgu}, $\psi$ induces a bijection 
$$\phi \colon \mathcal{F}(P) \to \mathcal{F}(P')$$ 
which also defines a face preserving homeomorphism $\bar \phi \colon P \to P'$. Moreover, by the proof of Theorem \ref{thm_equiv=SR}, we have a bijection between two sets of equivariant Thom classes of $M$ and $M'$ up to sign.

Since $\psi$ is an algebra isomorphism, the algebra structures are compatible with $\psi$, i.e., 
$\pi'^\ast=\psi\circ \pi^\ast$, where $\pi^\ast$ and $\pi'^\ast$ are defined in \eqref{eq_algebra_str} for $M$ and $M'$, respectively. Then, 
$$\pi'^\ast(u)= \sum_{j=1}^m \left<  \xi'_j, u \right> \tau'_j$$ and 
$$
\psi (\pi^\ast(u))= \sum_{j=1}^m \left<  \xi_j,  u \right> \psi (\tau_j) = \sum_{j=1}^m \epsilon_j \big\langle \xi_{\phi(j)}, u \big\rangle \tau'_{\phi(j)}
$$
for some $\epsilon_j=\pm 1$ and $m\colonequals |\mathcal{F}(P)|=|\mathcal{F}(P')|$. Here, we may regard $\phi$ as a permutation on $\{1, \dots, m\}$.  Hence, we get $\epsilon_j \langle \xi_{\phi(j)}, u \rangle = \langle  \xi'_j, u\rangle$ for each $j$. 

Let $J:= \{ j \mid \epsilon_j=-1\}$ and define an automorphism $\delta \colon T^m \to T^m$ 
such that $\delta$ sends $j$-th coordinate of  $(t_1, \dots, t_m)\in T^m$ to its conjugate whenever $j\in J$. This implies that $\delta( \exp (\ker \xi) )= \exp (\ker \xi')$. Hence, we have the following commutative diagram
\begin{equation*}
\begin{tikzcd}
\mathcal{Z}_P = (T^m \times P) /_{\sim_z} \arrow{r}{\Psi} \arrow{d} & (T^m \times P')/_{\sim_z}=\mathcal{Z}_{P'}\arrow{d}\\ 
\mathcal{Z}_P/ \exp(\ker \xi) \arrow{r}{\widetilde\Psi} & \mathcal{Z}_{P'}/ \exp (\ker \xi'),
\end{tikzcd}
\end{equation*}
where $\Psi$ is a $\delta$-equivariant homeomorphism induced from $\delta\times \bar{\phi}$. Therefore, $\widetilde{\Psi}$ is a weakly equivariantly homeomorphism with respect to an isomorphism between $T^m/ \ker \xi$ and $T^m/ \ker \xi'$ induced from $\delta$. Hence, the result follows in this case. 

\vspace{0.1cm}
\noindent\textbf{Case 2}. Now, we consider the case where $\rk(\xi) < n+k$. In this case, by Proposition \ref{prop_rank_n_n+1_classification}, $M$ is equivariantly homeomorphic to $T^{n+k - \rk(\xi)} \times N$ for some locally $r$-standard $T$-manifold $N$ with a surjective hyper characteristic function where $0 \leq r < k$. The torus $T$ is also decomposed into $T^{n+k-\rk(\xi)} \times T_\xi^{\rk(\xi)}$ accordingly. Observe that 
\begin{align*}
H^\ast_T(M)&=H^\ast(ET\times_T (T^{n+k - \rk(\xi)} \times N)) \\
&= H^\ast((ET^{n+k - \rk(\xi)} \times_{T^{n+k - \rk(\xi)}}T^{n+k - \rk(\xi)})\times (ET_\xi^{\rk(\xi)} \times_{T_\xi^{\rk(\xi)}} N))\\
&=H^\ast_{T_\xi^{\rk(\xi)}}(N).
\end{align*}
We note that $H^\ast(BT_\xi^{\rk(\xi)})$-algebra structure on $H^\ast_{T_\xi^{\rk(\xi)}}(N)$ is induced from the $H^\ast(BT)$-algebra structure on $H^\ast_T(M)$ and the decomposition 
$$H^\ast(BT)\cong H^\ast(BT^{n+k-\rk(\xi)})\otimes H^\ast(BT_\xi^{\rk(\xi)}),$$ 
where $H^\ast(BT^{n+k-\rk(\xi)})$ acts trivially on $H^\ast_{T_\xi^{\rk(\xi)}}(N)$. 

A weak isomorphism $\psi \colon H^\ast_T(M) \to H^\ast_T(M')$ induces a weak isomorphism 
$$\psi' \colon H^\ast_{T_\xi^{\rk(\xi)}}(N) \to H^\ast_{T_{\xi'}^{\rk(\xi')}}(N')$$
between two locally $r$-standard $T$-manifolds $N$ and $N'$, which are now in \textbf{Case 1}.  Hence, $N$ and $N'$ are weakly homeomorphic, which implies that $M$ and $M'$ are also weakly homeomorphic. 
\end{proof}

We note that Theorem \ref{thm_main_rigidity} generalizes the result of \cite{Mas} which deals with quasitoric manifolds.

\subsection*{Acknowledgements}
The authors are grateful to IBS--CGP for the hospitality in July 2019 and grateful to KAIST and IIT-Madras for supporting their visits. They are also grateful to Anthony Bahri,  Mikiya Masuda and Dong Youp Suh for helpful comments. The authors thank the anonymous referee for  helpful comments to improve the manuscript. 

The first author is supported by MATRICS grant MTR/2018/000963 of SERB India and International office of IIT Madras.
The second author has been supported by Basic Science Research Program through the National Research Foundation of Korea (NRF) funded by the Ministry of Education (NRF-2018R1D1A1B07048480) and a KIAS Individual Grant (MG076101) at Korea Institute for Advanced Study.

%
%\bibliographystyle{alpha}
%\bibliography{bibliography}
%

\end{document}